\documentclass[12pt,leqno,a4paper]{amsart}
\usepackage{amssymb,enumerate}
\newcommand{\QQ}{{\mathbb{Q}}}

\newcommand{\bG} {\mathbf G}
\newcommand{\bT} {\mathbf T}

\newcommand{\cE} {\mathcal E}

\newcommand{\Aut}{{\operatorname{Aut}}}

\newcommand{\Out}{{\operatorname{Out}}}
\newcommand{\Soc}{{\operatorname{soc}}}
\newcommand{\SL}{{\operatorname{SL}}}
\newcommand{\PSL}{{\operatorname{PSL}}}
\newcommand{\PSU}{{\operatorname{PSU}}}

\newcommand{\tw}[1]{{}^{#1}\!}

\let\eps=\epsilon

\def\irr#1{{\rm Irr}(#1)}
\def\cent#1#2{{\bf C}_{#1}(#2)}
\def\syl#1#2{{\rm Syl}_#1(#2)}
\def\norm#1#2{{\bf N}_{#1}(#2)}
\def\oh#1#2{{\bf O}_{#1}(#2)}
\def\Oh#1#2{{\bf O}^{#1}(#2)}
\def\sbs{\subseteq}

\newtheorem{thm}{Theorem}[section]
\newtheorem{lem}[thm]{Lemma}

\newtheorem*{thmA}{Theorem A}

\theoremstyle{definition}

\theoremstyle{remark}

\begin{document}

\title[Brauer's Height Zero Conjecture for Principal Blocks]{Brauer's Height Zero Conjecture for Principal Blocks}

\date{\today}

\author{Gunter Malle}
\address{FB Mathematik, TU Kaiserslautern, Postfach 3049,
        67653 Kaisers\-lautern, Germany.}
\email{malle@mathematik.uni-kl.de}
\author{Gabriel Navarro}
\address{Departament de Matem\`atiques, Universitat de Val\`encia, 46100
        Burjassot, Val\`encia, Spain.}
\email{gabriel@uv.es}

\begin{abstract}
We prove {\it the other half} of Brauer's Height Zero Conjecture in the case of
principal blocks.
\end{abstract}

\thanks{The first author gratefully acknowledges financial support by SFB TRR
195. The research of the second author is supported by Ministerio de Ciencia e
Innovaci\'on PID2019-103854GB-I00 and FEDER funds. We thank Yanjun Liu,
Lizhong Wang, Wolfgang Willems and Jiping Zhang for e-mails prompting our
interest in the subject of this paper, and Yanjun Liu for pertinent questions
on an earlier version.}

\keywords{Brauer's Height Zero Conjecture, principal blocks}

\subjclass[2010]{20C20}

\maketitle


\section{Introduction}   \label{sec:intro}
 
Richard Brauer's famous Height Zero Conjecture \cite{Br55} from 1955 proposes
that if $p$ is a prime, $G$ is a finite group, and $B$ is a Brauer $p$-block of
$G$ with defect group $D$, then $D$ is abelian if and only if all the irreducible
complex characters of $G$ in $B$ have height zero. For abelian $D$, the proof
was finally completed in \cite{KM13}.
The other direction was proven to hold assuming that all simple groups
satisfy the \emph{inductive Alperin--McKay condition} (see Navarro--Sp\"ath
\cite{NS14}). Prior to this, the conjecture was shown for $p=2$ and blocks of
maximal defect by Navarro--Tiep \cite{NT}.
\medskip
 
The inductive Alperin--McKay condition (Sp\"ath \cite{S13}) is a natural but 
difficult-to-check statement on blocks of simple groups which involves action of
automorphisms and cohomology. In this paper, we take a step back and conduct a
direct reduction to almost simple groups of Brauer's conjecture for principal
blocks. Then we handle these groups via the classification of finite simple
groups, hence proving the conjecture in this case.
\medskip
 
We denote by $B_0(G)$ the principal $p$-block of $G$, and by $\irr{B_0(G)}$ the
set of complex irreducible characters of $G$ in $B_0(G)$.
 
\begin{thmA}
 Let $p$ be a prime and let $G$ be a finite group. Then the Sylow $p$-subgroups
 of $G$ are abelian if and only if $p$ does not divide $\chi(1)$ for all
 $\chi \in \irr{B_0(G)}$.
\end{thmA}

In Problem 12 of his famous list \cite{Br63} from 1963 Brauer asked if
the character table of a finite group $G$ detects if $G$ has abelian Sylow
$p$-subgroups. Although this has been previously solved in \cite{KimSa}
(theoretically) and \cite{NST} (with an algorithm), Theorem~A gives the solution
to Problem 12 that surely Brauer had in mind.
\medskip

A crucial ingredient in our proof is the result by Kessar--Malle \cite{KM17}
that Brauer's Height Zero Conjecture holds for all quasi-simple groups.

At several points of the proof we do use special properties of principal blocks,
such as the Alperin--Dade theory of isomorphic blocks, which do not hold true for
general blocks.

\section{Reduction to Almost Simple Groups}

\begin{thm}   \label{thm:red}
 Brauer's Height Zero Conjecture is true for principal $p$-blocks if it is true
 for the principal $p$-block of every almost simple group $S$ such that
 $S/\Soc(S)$ is a $p$-group.
\end{thm}

\begin{proof}
Let $G$ be a finite group, and let $P\in \syl pG$. We assume that $p$ does not
divide the degrees of the complex irreducible characters in $B_0(G)$. We argue
by induction on $|G|$ that $P$ is abelian.

Let $1<N<G$ be a normal subgroup of $G$. Since
$\irr{B_0(G/N)} \sbs \irr{B_0(G)}$, we have that $G/N$ has abelian
Sylow $p$-subgroups, by induction. If $\theta \in \irr{B_0(N)}$,
then there exists $\chi \in \irr{B_0(G)}$ such that $\theta$ is an irreducible
constituent of the restriction $\chi_N$ (by \cite[Thm.~9.4]{N}).
Since $\theta(1)$ divides $\chi(1)$, we have that $N$ has abelian Sylow
$p$-subgroups, by induction.  In particular, we have that $\Oh{p'} G=G$.
If $G$ has two different minimal normal subgroups $N_1$ and $N_2$,
then $G$ is isomorphic to a subgroup of $G/N_1 \times G/N_2$ that has abelian
Sylow $p$-subgroups, and therefore $G$ has abelian Sylow $p$-subgroups.
So we may assume that $G$ has a unique minimal normal subgroup $N$.

Since $\irr{B_0(G)}= \irr{B_0(G/\oh{p'}G}$ (by \cite[Thm.~6.10]{N}), we may
assume by induction that $\oh{p'} G=1$. 

Recall that if a finite group $H$ has abelian Sylow $p$-subgroups and
$\oh{p'} H=1$, then $\Oh{p'} H=\Oh {pp'} H \times \oh pH$, where $\Oh {pp'}H=\Oh p{\Oh{p'} H}$ is either
trivial or the direct product of non-abelian simple groups of order divisible by
$p$, with abelian Sylow $p$-subgroups, and $\oh pH$ is abelian.
(See for instance, \cite[Thm.~2.1]{NT2}.) Hence, if $M$ is any proper
normal subgroup of $G$, then $\Oh{p'} M$ is a direct product of $\oh pM$ and
$\Oh{p'p} M$, and this latter is either trivial or a direct product of simple
groups of order divisible by $p$. Also, again, if $L/N=\oh{p'}{G/N}$, 
we have that $G/L=X/L \times Y/L$, where $X/L$ is a direct product of simple
groups of order divisible by $p$ (or trivial), and $Y/L$ is a $p$-group. 

Suppose first that $N$ is a central $p$-group. Let $\tau\in \irr P$.
Write $\tau_N=e \lambda$, where $\lambda \in \irr N$. By \cite[Thm.~9.4]{N}, let
$\gamma_0 \in \irr{B_0(\norm GP})$ over $\tau$. By \cite[Cor.~6.4]{N} and
Brauer's Third Main theorem, there is some $\hat\tau \in \irr{B_0(G)}$ over
$\gamma_0$. Thus $\hat\tau$ has $p'$-degree. Hence, $\hat\tau_P$ contains some
$p'$-degree character $\gamma$, and $\gamma_N=\lambda \in \irr N$ by
\cite[Cor.~11.29]{Is06}. Since $P/N$ is abelian by induction, by Gallagher's
Corollary 6.17 of \cite{Is06} we have $\tau=\gamma \rho$ for some linear
$\rho\in\irr{P/N}$. Hence, all irreducible characters of $P$ are linear, 
and we are done in this case. 

Suppose next that $N$ is a $p$-group. Let $C=\cent GN$. Then all characters in
$\irr{G/C}$ are in the principal $p$-block of $G$ (by \cite[Lemma 3.1]{NT} and
Brauer's third main theorem). By the It\^o--Michler
theorem \cite{Mi86} and our hypothesis, we have that $G/C$ has a normal
Sylow $p$-subgroup. Since $\Oh{p'}G=G$, we have that $G/C$ is a $p$-group.
By the previous paragraph, we may assume that $C<G$. Hence, $\Oh{p'} C$ is a
direct product of $\Oh{p'p} C$ and $\oh pC$. These are normal subgroups of $G$,
since they are characteristic in $C$. Since $G$ has a unique minimal normal subgroup and
$\oh{p}C>1$, we conclude that $\Oh{p'} C$ is a $p$-group. In this case, $G$ is
$p$-solvable and $\irr{B_0(G)}=\irr G$ (by \cite[Thm.~10.20]{N}, for instance).
Again by the It\^o--Michler theorem, we conclude that $G$ has a normal Sylow
$p$-subgroup. Hence $G=P$ as $\Oh{p'} G=G$, and again we are done. 
 
Thus we may assume that $\oh pG=1$. Hence, $N={\bf F^*}(G)$ is a direct product
of non-abelian simple groups of order divisible by $p$. Also, if $M$ is any
proper normal subgroup of $G$, then $\Oh{p'} M$ is a direct product of
non-abelian simple groups of order divisible by~$p$.

Let $Q=P\cap N \in \syl pN$. We have that the irreducible characters of
$\irr{G/N\cent GQ}$ are in the principal block (again by \cite[Lemma 3.1]{NT}
and Brauer's third main theorem), so $G/N\cent GQ$ is a $p$-group by the
It\^o--Michler theorem, and using that $\Oh{p'} G=G$. Then $L \sbs N\cent GQ$.
Hence, $L=N\cent LQ$.

Since $X$ is not a $p'$-group (because $\oh{p'}G=1$), we have that
$N\sbs\Oh{p'}X$. Assume that $X<G$. Since $\Oh{p'}{X}$ is a direct product of
non-abelian simple groups and $N \unlhd\Oh{p'}{X}$, we have
$\Oh{p'}{X}=N \times U$, by elementary group theory. But then
$U=1$, because $N={\bf F^*}(G)$ and $\cent GN \sbs N$. Thus $\Oh{p'}X=N$.
Since $X/L$ is a direct product of non-abelian simple groups of order divisible
by $p$ (or trivial) and $N\sbs L\sbs X$, we deduce that $X=L$ and therefore
$Y=G$. Thus $G/N$ has a normal $p$-complement $L/N$. We claim that $PN$
satisfies the hypotheses.
Let $\tau \in \irr{B_0(PN)}$. Let $\gamma \in \irr N$ be under $\tau$. Then
$\gamma$ lies in the principal block of $N$. Since $\tau$ lies under some
character in the principal block of $G$, we have that $\gamma$ has $p'$-degree
using the hypothesis. By the Alperin--Dade theory on isomorphic blocks
(\cite{Al,D}), $\gamma$ has a canonical extension $\hat\gamma\in\irr L$ in the
principal block of $L$.  Now,
$\hat\gamma$ lies under some $\chi \in \irr{B_0(G)}$. By hypothesis, $\chi$ has
$p'$-degree, and therefore $\chi_L=\hat\gamma$ (by \cite[Cor.~11.29]{Is06}), and $\hat\gamma$ is
$G$-invariant. In particular, $\gamma$ is $G$-invariant too. By the Isaacs character correspondence (\cite[Cor.~4.2]{IS})
there is a bijection
$$\irr{G|\hat\gamma} \rightarrow \irr{PN|\gamma} \, ,$$
given by restriction. Since $\tau$ lies over $\gamma$, let $\hat\tau\in \irr G$
over $\hat\gamma$ that restricts to $\tau$. Since $G/L$ is a $p$-group and $\hat\tau$ lies over $\hat\gamma$,
we have
that $\hat\tau$ is in the principal $p$-block of $G$ (\cite[Cor.~9.6]{N}), thus $\tau$ has
$p'$-degree. So, by induction, we may assume that $G=NP$. Let $S$ be normal
simple in $N$ and write $N=S^{x_1}\times\cdots\times S^{x_t}$, where
$P=\norm PSx_1\cup\ldots\cup\norm PS x_t$. Now, let $1\ne\tau\in \irr{B_0(S)}$,
and consider $\theta=\tau\times 1\times\cdots\times 1\in\irr{B_0(N)}$.
Let $\chi\in\irr {B_0(G)}$ over $\tau$.
Then $\chi$ has $p'$-degree, so $\chi_N=\theta$. Then $\theta$ is invariant. But
this is impossible unless $t=1$. Thus we obtain that $G=NP$ is almost simple,
with $\Soc(G)=N=S$, and $G/N$ a $p$-group.

We are left with the case that $X=G$. Hence $G/L$ is a direct product of
non-abelian simple groups of order divisible by $p$. Since $L \sbs N\cent GQ$ 
and $G/N\cent GQ$ is a $p$-group, we have that $G=N\cent GQ$.
Write $N=S^{z_1}\times\cdots\times S^{z_u}$, where $S$ is simple non-abelian,
and $\{S^{z_1}, \ldots, S^{z_u}\}$ are the different $G$-conjugates of $S$. 
Now, $Q\cap S\in\syl p S$. If
$1\ne x\in Q\cap S$ and $c\in\cent GQ$, then $x^c=x\in S$. In particular,
$c\in\norm GS$, and therefore $S\unlhd G$. Hence $S=N$. By the Schreier
theorem, we have that $G/N$ is solvable, but this is not possible.
\end{proof}

As pointed out by the referee, the proof of Theorem \ref{thm:red} could be
shortened somewhat by using the theory of the so called $p^*$-groups
(see \cite{Wi86}).

\section{Almost Simple Groups}
We now deal with almost simple groups with $p$-automorphisms.
The first observation follows from the classification of finite simple groups.
We refer the reader to \cite[\S2.5]{GLS} for a description of the outer
automorphism groups of the groups in question.

\begin{lem}   \label{lem:aut}
 Let $S$ be non-abelian simple with abelian Sylow $p$-subgroups and suppose that
 $S$ has a non-trivial outer automorphism $\sigma$ of $p$-power order. Then $S$
 is of Lie type and either $\sigma$ is a field automorphism and $p$ does not
 divide the order of the group of outer diagonal automorphisms of $S$, or one of
 the following holds:
 \begin{itemize}
  \item[\rm(1)] $p=2$ and $S=\PSL_2(q)$ with $q\equiv3,5\pmod8$;
  \item[\rm(2)] $p=3$ and $S=\PSL_3(q)$ with $q\equiv4,7\pmod9$; or
  \item[\rm(3)] $p=3$ and $S=\PSU_3(q)$ with $q\equiv2,5\pmod9$.
\end{itemize}
\end{lem}

\begin{proof}
If $S$ is sporadic or alternating, then $\Out(S)$ has 2-power order, so $p=2$.
The only such $S$ with abelian Sylow 2-subgroups is $J_1$, but $\Out(J_1)=1$. 
\par
Thus, by the classification of finite simple groups, $S$ is of Lie type. Here,
if $p=2$, then by the well-known classification by Walter \cite{Wa69}, $S$ is
one of $\SL_2(2^f)$ with $f\ge2$, $\PSL_2(q)$ with $q\equiv3,5\pmod8$, or
$^2G_2(3^{2f+1})$ with $f\ge1$. The groups not occurring in the conclusion do
not have outer automorphisms that are not field automorphisms. Now assume that
$p>2$. Here $\sigma$ is a product of diagonal, graph and field automorphisms.
Non-trivial graph automorphisms of order $p\ge3$ only exist for groups of type
$D_4(q)$, but here Sylow 3-subgroups are non-abelian (since, for example,
$\PSL_4(q)$ is a section). Non-trivial outer
diagonal automorphisms of order $p\ge3$ only exist for groups of types $A_n$
and $E_6$. For $S=E_6(q)$ or $\tw2E_6(q)$ these have order~3, but Sylow
3-subgroups of $S$ are non-abelian. For $S=\PSL_n(q)$ or $\PSU_n(q)$ the outer
diagonal automorphisms have order dividing $n$, but Sylow $p$-subgroups of $S$
for $3\le p|n$ are non-abelian, unless $n=p=3$. The latter cases appear in
our conclusion.
\end{proof}

We are thus reduced to considering groups of Lie type. For these, we require
some notions and results from Deligne--Lusztig theory. Let $\bG$ be a simple
algebraic group of adjoint type and $F:\bG\to\bG$ a Steinberg endomorphism, with
finite group of fixed points $G=\bG^F$. Let $(\bG^*,F)$ be in duality
with~$\bG$. Any field automorphism $\sigma$ of $G$ is induced by a Frobenius
map $F_1:\bG\to\bG$ commuting with $F$ and the corresponding Frobenius map
$F_1:\bG^*\to\bG^*$ induces a field automorphism $\sigma^*$ of $G^*=\bG^{*F}$
(see \cite[\S5]{Tay18}).

\begin{lem}   \label{lem:non-fix}
 Let $G$ be as above of adjoint Lie type with abelian Sylow $p$-subgroups for a
 non-defining prime $p$, and $\sigma$ a field automorphism of $G$ of $p$-power
 such that $G\langle\sigma\rangle$ has non-abelian Sylow $p$-subgroups. Then
 $\sigma^*$ does not fix all classes of $p$-elements in~$G^*$.
\end{lem}

\begin{proof}
First note that our assumptions on $G$ force $p>2$. Let $F_1:\bG\to\bG$ be a
Frobenius map inducing $\sigma$ and commuting with~$F$. Since a Sylow
$p$-subgroup $P$ of $G$ is abelian, it is contained in an $F$-stable maximal
torus $\bT$ of~$\bG$ (see \cite[Thm.~25.14]{MT}), which we may assume to be
$F_1$-stable. In fact, if $F$ is not very twisted, $\bT$ contains a Sylow
$d$-torus $\bT_d$ of $\bG$, where $d$ is the order of $q$ modulo~$p$ and $q$ is
the underlying field size of $G$, while in the case of Suzuki and Ree groups, it
contains a Sylow $\Phi$-torus which again we denote $\bT_d$, where $\Phi$ is the
minimal polynomial over $\QQ(\sqrt{2})$ respectively $\QQ(\sqrt{3})$ of a
primitive $d$th root of unity (see \cite[3F]{BM}). Since $G$ and $G^*$,
as well as $\bT^F$ and $\bT^{*F}$, have the same order, $\bT^{*F}$ contains a
Sylow $p$-subgroup of~$G^*$, which is hence also abelian. Let $\sigma^*$ be the
dual automorphism of~$G^*$ induced by $F_1:\bG^*\to\bG^*$ which we may and will
assume to stabilise $\bT^*$. Let $P^*$ be the Sylow $p$-subgroup of $G^*$
contained in $\bT^{*F}$.
\par
By assumption, $\sigma$ acts non-trivially on~$P$. Thus the centraliser
$G_1:=C_G(\sigma)$ of $\sigma$, a subfield subgroup of~$G$, does not contain a
Sylow $p$-subgroup of $G$. Since $\sigma^*$ is dual to $\sigma$,
$G_1^*:=C_{G^*}(\sigma^*)$ has the same order as $G_1$, so $\sigma^*$ does also
not centralise~$P^*$. Now $G^*$-fusion in $P^*$ is controlled by the relative
Weyl group $W=N_{G^*}(\bT_d^*)/C_{G^*}(\bT_d^*)$ of the Sylow torus $\bT_d^*$,
by \cite[Prop.~5.11]{Ma07}. Let $q_1$ be the underlying field size of
$\bG^{F_1}$, so $q=q_1^{p^a}$ for some $a\ge1$. Since $q$ and $q_1$ have the
same order~$d$ modulo~$p$, the relative Weyl groups $W$ in $G^*$ and
$W_1=N_{G_1^*}(\bT_d^*)/C_{G_1^*}(\bT_d^*)$ in $G_1^*$ agree, that is,
$\sigma^*$ commutes with the action of $W$ on $P^*$. Now $W$ has order prime
to $p$ as $\bT_d^{*F}$ contains a Sylow $p$-subgroup of $\bG^{*F}$, while
$\sigma^*$ has some non-trivial orbit, of length divisible by~$p$, on $P^*$.
So $\sigma^*$ induces additional fusion on elements of $P^*$ and
our claim follows.
\end{proof}

\begin{thm}   \label{thm:as}
 Let $p$ be a prime, $S$ be a non-abelian simple group and $S\le A\le\Aut(S)$
 such that $A/S$ is a $p$-group. Assume that all characters in the principal
 $p$-block of $A$ have degree prime to $p$. Then the Sylow $p$-subgroups of $A$
 are abelian.
\end{thm}

\begin{proof}
We will assume that $A$ has non-abelian Sylow $p$-subgroups and argue by
contradiction. The case $A=S$ is the main result of \cite{KM17}, so we may
assume that $A\ne S$, $S$ has abelian Sylow $p$-subgroups, and we have to
exhibit a character in the principal $p$-block of $S$ that is not $A$-invariant.
\par
By Lemma~\ref{lem:aut} we may assume that $S$ is of Lie type.
We first discuss the groups showing up in Lemma~\ref{lem:aut}(1)--(3).
For $S=\PSL_2(q)$ the outer automorphism group is generated by diagonal and
field automorphisms. As $q\not\equiv1\pmod8$, $S$ does not have even order field
automorphisms. Now the principal 2-block of $S$ contains two irreducible
characters of degree $(q+\eps)/2$, where $q\equiv\eps\pmod4$, $\eps\in\{\pm1\}$,
that are not invariant under the diagonal automorphism (see e.g.
\cite[Lemma~15.1]{IMN}). For $\PSL_3(q)$ with $3||(q-1)$,
the diagonal automorphism of order~3 moves the three characters of degree
$(q+1)(q^2+q+1)/3$, and similarly for $\PSU_3(q)$ with $3||(q+1)$ the three
characters of degree $(q-1)(q^2-q+1)/3$, (see e.g. \cite[Lemmas~3.5,
3.11]{Ma08}). For the given congruences, $S$ does not have field automorphisms
of 3-power order. Thus, we may now assume that $S$ is not one of the exceptions
in Lemma~\ref{lem:aut}.
\par
If $p$ is the defining characteristic of $S$, then as $S$ has abelian Sylow
$p$-subgroups, we have $S=\PSL_2(q)$ with $q=p^f>p$. Here, all irreducible
characters apart from the Steinberg character lie in the principal block. Let
$\chi$ be a Deligne--Lusztig character of degree~$q-1$ corresponding to a
semisimple element $s$ generating the non-split maximal torus $T$ of order
$(q+1)/(2,q-1)$ in the dual group. A field automorphism $\sigma$ of $S$
centralises the corresponding group over a subfield; since no proper subfield
group contains a cyclic subgroup of order $(q+1)/(2,q-1)$, $\sigma$ must act
non-trivially on~$T$. Thus, $\chi$ is not $\sigma$-stable.
\par
So $p$ is not the defining characteristic of $S$. Let $\bG$ and $G=\bG^F$ be as
above such that $S=[G,G]$. This is possible unless $S$ is the Tits simple group,
which has no non-trivial field automorphisms. Any field automorphism of $S$
extends to $G$. Now note that Sylow $p$-subgroups of $G$ are also abelian,
since $G$ acts by diagonal automorphisms on $S$ and $p$ does not divide their
order. Then by Lemma~\ref{lem:non-fix} there is some conjugacy class of a
$p$-element $s\in G^*$ not fixed by $\sigma^*$. By \cite[Prop.~7.2]{Tay18} this
implies that the corresponding Lusztig series $\cE(G,s)$ is not fixed by
$\sigma$. In particular the semisimple characters in $\cE(G,s)$, which lie in
the principal $p$-block by \cite[Thm.]{CE94} are not fixed by $\sigma$. Since
$p$ does not divide
the order of the group of outer diagonal automorphisms of $G$, the number of
characters of $S$ below $G$ is not divisible by $p$, hence there is some
character of $S$ moved by $\sigma$, and we conclude.
\end{proof}

Our main theorem now follows by combining Theorems~\ref{thm:red}
and~\ref{thm:as}.


\end{document}